\documentclass[11pt]{article}
\usepackage{inputenc}
\usepackage[english]{babel}
\usepackage{amssymb,amsmath}
\usepackage[top = 3cm, left = 3cm, right = 3cm, bottom = 2cm]{geometry}      					
            
\RequirePackage{hyphenat}						
\RequirePackage[shortlabels]{enumitem}
\RequirePackage{float}
\RequirePackage{footmisc}							
\RequirePackage{setspace}        
\RequirePackage{imakeidx}
\RequirePackage{hyperref}
\RequirePackage{graphicx}
\RequirePackage{caption}
\RequirePackage{subcaption}
\RequirePackage{csquotes}

\RequirePackage[nottoc]{tocbibind}                  
\RequirePackage{import}                                                   
\usepackage[backend=biber,style=alphabetic]{biblatex}
\DeclareNameAlias{default}{family-given}
\RequirePackage{amsthm}
\RequirePackage{mathtools}         
\RequirePackage{tensor} 		
\RequirePackage{stackrel}         
\RequirePackage{centernot}   
\RequirePackage{tikz-cd} 
\RequirePackage{extarrows}
\RequirePackage{fancyhdr}
\RequirePackage{mathrsfs}
\RequirePackage{esvect}

\def\th@definition{
	\thm@notefont{}	
	\normalfont  
}

\hypersetup{
    colorlinks=true,
    linkcolor=blue,
    filecolor=magenta,      
    urlcolor=cyan,
    citecolor=black,
    pdftitle={Overleaf Example},
    pdfpagemode=FullScreen,
    }

\newtheorem{theorem}{Theorem}[section]
\newtheorem{corollary}[theorem]{Corollary}
\newtheorem{lemma}[theorem]{Lemma}
\newtheorem{proposition}[theorem]{Proposition}
\newtheorem{definition}[theorem]{Definition}
\newtheorem{remark}[theorem]{Remark}

\newcommand{\NS}{\operatorname{NS}}
\newcommand{\Sym}{\operatorname{Sym}}
\newcommand{\N}{\operatorname{N}}
\renewcommand{\H}{\operatorname{H}}
\newcommand{\T}{\operatorname{T}}
\newcommand{\Aut}{\operatorname{Aut}}
\newcommand{\NE}{\operatorname{NE}}
\newcommand{\Amp}{\operatorname{Amp}}

\def \P{{\mathbb P}}
\def \Z{{\mathbb Z}}
\def \Q{{\mathbb Q}}
\def \R{{\mathbb R}}

\begin{document}

\title{An example of potential density on $Hilb^3$ of a K3 surface}
\author{E. Amerik, M. Lozhkin}
\maketitle

\section{Introduction}

An irreducible projective variety $X$ defined over a number field $K$ is said to have potentially dense rational points, or sometimes simply to be ``potentially dense (PD)'', if there is a finite extension $L$ of $K$ such that the set of $L$-points $X(L)$ is Zariski-dense in $X$. For example, $\P^n_K$ is potentially dense for any $K$; it follows that the unirational varieties are potentially dense. 

It is also not difficult to prove that the abelian varieties are potentially dense (see for example Prop. 3.1 of \cite{HT-abelian}, 
and also Prop. 2.11 of \cite{ABR} for an alternative, elementary and self-contained proof). On the contrary, by a celebrated theorem of 
Faltings, curves of geometric genus at least two are not potentially dense. It is generally believed and fits into 
Lang--Vojta's conjectural framework that the smooth projective varieties of general type are not potentially dense, whereas the varieties with negative or trivial canonical bundle are. Campana conjectures that potentially dense varieties are exactly the special varieties in the sense of \cite{C} (the class of special varieties includes all rationally connected varieties and all varieties of zero Kodaira dimension). But these conjectures are difficult to verify even in dimension two. 

For instance, we still do not know whether a general K3 surface is potentially dense. The potential density has been obtained by Bogomolov and Tschinkel for elliptic K3 surfaces, and also for K3 surfaces with infinite automorphism group (\cite{BT}). But no example of a potentially dense (or of a not potentially dense) K3 surface with Picard number one is available. Only one example of a family of simply-connected varieties with trivial canonical bundle, such that a sufficiently general member has Picard number one and is potentially dense, is known: such are the varieties of lines contained in cubics in $\P^5$ (\cite{AV}, \cite{ABR}). These varieties are deformations of the second punctual Hilbert scheme of certain K3 surfaces. This family is complete in the sense that it exhausts all smooth polarized deformations of its general member. So far, this example remains unique. 

Potential density of the $k$-th punctual Hilbert scheme $S^{[k]}$
of a K3 surface $S$ has been studied by Hassett and Tschinkel in \cite{HT-abelian}. Roughly speaking, 
they prove that $S^{[k]}$ is potentially dense when it admits (birationally) an abelian fibration. In particular,
when $S\subset \P^g$ is a K3 surface of genus $g$, then $S^{[g]}$ is potentially dense (indeed it is birational to the bundle
of $\operatorname{Pic}^g$ of its hyperplane sections).

In the note \cite{A}, the  first-named author proved\footnote{We should also give credit to \cite{OG}, where at some point the problem and the idea of proof is mentioned. At the time, Amerik was not aware of this work.} the potential density for $S^{[2]}$
of certain K3 surfaces $S$ admitting two embeddings as a quartic in $\P^3$, and chosen in such a way that $S$ has only finitely many automorphisms and no elliptic fibrations (otherwise $S$ itself is known to be potentially dense, so it is not interesting to study $S^{[2]}$). These properties are encoded in the N{\'e}ron-Severi lattice of $S$, which she took to be of rank two and with the intersection form $4x^2+14xy+4y^2$. K3 surfaces with this N{\'e}ron-Severi lattice form an open subset of a codimension-one subvariety in the space of quartics in $\P^3$. $S^{[3]}$ is potentially dense by \cite{HT-abelian}. To prove potential density of $S^{[2]}$, the idea was to use an automorphism similar to the one studied by Oguiso in \cite{O}. 

Recall that $S^{[2]}$ of a quartic surface $S$ has a natural involution $\iota$, the Beauville involution, sending an unordered pair of points $x_1+x_2$ to the residual pair of points in the intersection of $S$ and the line through $x_1$ and $x_2$. When $S$ contains no line, $\iota$ is a regular map. When $S$ admits two different embeddings as a quartic, the product $f=\iota_2\iota_1$ of the corresponding Beauville involutions is an automorphism of $S^{[2]}$ 
of infinite order. The idea of \cite{A} was to take a hyperplane section $C\subset S$ with one double point and iterate the surface $\Sigma=C*C\subset S^{[2]}$ (coming from $\operatorname{Sym}^2C$, see \cite{HT-abelian} for the precise definition). The surface $\Sigma$ is birationally abelian and therefore potentially dense, so the potential density of $X=S^{[2]}$ follows
once we know that the union of $f^k(\Sigma)$ is Zariski-dense. This in turn means that $\Sigma$ is not periodic and that the Zariski-closure of $\bigcup_k f^k(\Sigma)$ is not a divisor. Both statements are easily checked by a computation in cohomologies $H^2(X,\Q)\supset NS(X)$ and $H^4(X,\Q)=\operatorname{Sym}^2H^2(X,\Q)$.

The purpose of the present paper is to give a similar example in dimension six. Namely, we consider 
a K3 surface $S$ which admits two projective embeddings as an intersection of a quadric and a cubic in $\P^4$ (still with finite automorphism group and no elliptic fibration).
This leads to two Beauville involutions on $X=S^{[3]}$, but now these are only rational maps even for $S$ generic: the indeterminacy locus of each is a $\P^3$ formed by triples of colinear points on $S$. The product $f$ of two Beauville involutions is a
birational automorphism of $X$. We take for $\Sigma\subset S^{[3]}$ the indeterminacy locus of $\iota_2$, and we prove that the union of $f^k(\Sigma)$ is Zariski-dense. This is done in three steps. 

The proof that the Zariski closure of $\bigcup_k f^k(\Sigma)$ is not a divisor is just a computation in cohomology, which yields, exactly 
as in the 4-dimensional case, that no effective divisor can be $f$-invariant. We also need to prove that $\Sigma$ is not 
periodic, and that the Zariski closure of $\bigcup_k f^k(\Sigma)$ is not four-dimensional. But this time it is not so easy to show this by a cohomological computation, because of two reasons: the cohomology itself is more complicated, and especially,
$f^*$ is not a ring homomorphism anymore. The proof of non-periodicity in Section 6 is an inductive argument which uses the explicit description of the Mori cone of $X$ (Section \ref{mori}). The proof that the Zariski closure of the iterates of $\Sigma$ cannot be four-dimensional (Section \ref{4-dim}) uses symplectic geometry in an essential way. Our key tool is a generalization of Jouanolou-Ghys theorem by Correa, Maza and Soares \cite{CMS}.

Our main result is as follows.

\medskip

\noindent{\bf Theorem:}\label{th: main} Let $S$ be a K3 surface with N{\'e}ron-Severi lattice of rank 2, with intersection form $6x^2+16xy+6y^2$.
Then $X=S^{[3]}$ is potentially dense.

\section{The example} \label{example}
    Let $S$ be a K3 surface and $n > 1$ a positive integer. The N{\'e}ron-Severi lattice of the variety $X = S^{[n]}$ decomposes as the orthogonal direct sum with respect to the Beauville-Bogomolov quadratic form $q$ on $X$:
    \begin{equation} \label{NS-decomp}
        \NS(X) \cong \NS(S) \oplus \mathbb{Z}E,
    \end{equation}
    where $2E$ is the exceptional divisor of the Hilbert-Chow morphism $S^{[n]} \to \Sym^n S$, $q(E) = -2(n - 1)$ and $q\vert_{\NS(S)}$ coincides with the intersection form on $S$ (\cite{B1983}). The group generated by algebraic curve classes modulo homological equivalence $\N_1(X) \subset \H_2(X, \mathbb{Z})$ can be viewed, using $q$, as a subgroup of $\NS(X, \mathbb{Q}) = \NS(X) \otimes \mathbb{Q}$.

    Take a degree 6 K3 surface $S$ of Picard number $\rho(S) = 2$, with intersection form \begin{equation} \label{eq: form}
        b(x, y) = 6x^2 + 16xy + 6y^2
    \end{equation} in some coordinates $(x,y)$.
    Throughout this paper, it is more convenient for us to work with column vectors, so we write $(x, y)^{\T}$ meaning the transposition of the corresponding row vector $(x, y)$.

\begin{proposition}
The surface $S$ has $(-2)$-curves. This surface is not elliptic. The variety $X$ is not rationally fibered in abelian threefolds.
\end{proposition}
\begin{proof}
    The class $(2, -1)^{\T}$ is an example of a
    $(-2)$-class, so the surface $S$ has $(-2)$-curves. If the surface $S$ is elliptic or the variety $X$ is rationally fibered in abelian threefolds, then by \cite{AC2008} the form $q$ represents zero, i.e. $3x^2 + 8xy + 3y^2 - 2z^2 = 0$ for some $x, y, z \in \mathbb{Z}$. If $2\nmid x$, then counting modulo 8 yields the contradiction. If $2\mid x$, then both $y$ and $z$ are even, so applying the same argument to the triple $(x/2, y/2, z/2)$ implies $x = y = z = 0$. \qed
\end{proof}

\begin{proposition}
    The group $|\Aut(S)|$ is finite.
\end{proposition}
\begin{proof}
    It is proven in \cite{PS1971}, section 7, that a K3 surface with Picard number 2 has an infinite automorphism group if and only if its intersection form represents both $-2$ and $0$. \qed
\end{proof}

\begin{proposition}
    There are two very ample classes $h_1$ and $h_2$ on the surface $S$ with self-intersection 6.
\end{proposition}
\begin{proof}
    We repeat the argument of \cite{A}. Denote by $h_1$ and $h_2$ the classes of line bundles corresponding to the classes $(1, 0)^{\T}$ and $(0, 1)^{\T}$. Observe that there are no classes $v = (x, y)^{\T}$ with $v^2 = -2$ and $v \cdot h_1 = 0$. Indeed, if such a class $v$ exists,  then $6x + 8y = 0$ and $6x^2 + 16xy + 6y^2 = -2$. Using the first equation we get that $3 \mid y$ which contradicts the second equation. Therefore, applying some sequence of Picard-Lefschetz reflections if necessary, we may assume that the class $h_1$ is ample. 

    Now, we show that the class $h_2$ is automatically ample too. It is enough to verify that intersections of the classes $h_1$ and $h_2$ with any $(-2)$-class have the same sign. Let $v = (x, y)^{\T}$ be a $(-2)$-class intersecting $h_1$ positively, i.e. $3x^2 + 8xy + 3y^2 = -1$ and $3x + 4y > 0$. Denoting $t = 3x + 4y$ we get that $t^2 - 7y^2 = -3$ and $t > 0$. The equality $t^2 = 7y^2 - 3$ implies $4t > 7y$, so $v \cdot h_2 = \frac{1}{3}(4t - 7y) > 0$. Thus, the class $h_2$ is ample.

    It remains to show that classes $h_1$ and $h_2$ are in fact very ample. Since for any $(-2)$-curve class $L$, we have $L \cdot h_1 \notin \{0, 1\}$, then by \cite{SD1974}, 2.7 the divisor $h_1$ has no fixed components. Any divisor having positive self-intersection without fixed components is either very ample or hyperelliptic. By \cite{SD1974}, 5.4 the latter is not the case, since the form $q$ does not represent 0 and there is no curve $C'$ with $2[C'] = h_1$. \qed
\end{proof}

Without loss of generality, we assume in what follows that $h_1=(1, 0)^{\T}$ and $h_2=(0, 1)^{\T}$ are very ample. 

\begin{lemma}
    The classes $2h_1 - h_2$ and $2h_2 - h_1$ correspond to $(-2)$-curves.
\end{lemma}
\begin{proof}
    It is enough to prove the statement for the class $2h_1 - h_2$. Since $(2h_1 - h_2)^2 = -2$, either the class $2h_1 - h_2$ or the class $-(2h_1   - h_2)$ is effective. The class $h_1$ is ample and $(2h_1 - h_2) \cdot h_1 = 4$, so the class $2h_1 - h_2$ is effective and corresponds to some curve $C$. Suppose that the curve $C$ is not irreducible or not reduced. Then we can write $$[C] = [C'] + [C'']$$ for some effective classes $[C'], [C'']$ with $\left(C'\right)^2 = -2$. Since $$4 = h_1 \cdot [C] > h_1 \cdot [C'] > 0$$ and $2 \mid \left(h_1 \cdot [C']\right)$, then $h_1 \cdot [C'] = 2$. Now, if we write $[C'] = xh_1 + yh_2$, then this condition reads as $3x + 4y = 1$, and the condition $\left(C'\right)^2=-2$ is equivalent to $(3x + 4y)^2 - 7y^2 = -3$. The equation $7y^2 = 4$ has no integer solutions, therefore, the curve $C$ is reduced and irreducible. \qed
\end{proof}

\begin{proposition}
	There exist K3 surfaces $S$ defined over a number field, with $\rho(S)=2$ and the intersection form $b$ as in \eqref{eq: form}.
\end{proposition}
\begin{proof}
    Such a K3 surface over $\mathbb{C}$ is a general member of a component of the Noether-Lefschetz locus of the family of intersections of a quadric and a cubic in $\mathbb{P}^4$. This component is an algebraic subvariety defined over a number field. Now the only problem is that it could, apriori, happen that every surface from this locus which is defined over a number field has higher Picard number; but this is ruled out by \cite{MP}, which shows that in any family (defined over a number field) of smooth projective varieties, most members over a number field have the same N{\'e}ron-Severi group as the general member. \qed
\end{proof}

\section{Birational automorphism and its invariant divisors} \label{inv}

The classes $h_1$ and $h_2$ define two embeddings $S \subset \mathbb{P}^4$ as an intersection of a quadric and a cubic. Two embeddings induce two Beauville involutions $\iota_1, \iota_2: S^{[3]} \dashrightarrow S^{[3]}$, see \cite{B1982}. Recall that a Beauville involution of $S^{[3]}$ sends a general triple of points to the residual triple in the intersection of $S$ and the plane through the triple. Denote by $H_1, H_2 \in \NS(X) \simeq \NS(S) \oplus \mathbb{Z}E$ the classes corresponding to $h_1, h_2 \in \NS(S)$ respectively. The induced morphism $\iota_k: \NS(X) \to \NS(X)$ $(k \in \{1, 2\})$ is opposite to the reflection in the hyperplane $(H_k - E)^{\perp}$ (\cite{D1984}). In other words, the action of the homomorphism $\iota_k$ is given by $$\iota_k^*(L) = -L + 2\frac{q(L, H_k - E)}{q(H_k - E, H_k - E)}(H_k - E) = -L + q(L, H_k - E)(H_k - E).$$
In the basis $\{H_1, E, H_2\}$ the operators $\iota_1^*, \iota_2^*$ are given by matrices  
\begin{equation} \label{eq: matrices}
M_1 = \begin{pmatrix}
    5 & 4 & 8 \\
    -6 & -5 & -8 \\
    0 & 0 & -1 
\end{pmatrix}, \qquad M_2 = \begin{pmatrix}
    -1 & 0 & 0 \\
    -8 & -5 & -6 \\
    8 & 4 & 5 
\end{pmatrix}.
\end{equation}

The following construction is described in \cite{B1982}. Denote by $\tilde{\Pi}_k$ the indeterminacy set of the involution $\iota_k$. The set $\tilde{\Pi}_k$ consists of length 3 subschemes contained in a line. The image of $S$ under the embedding given by $h_k$ is defined by the intersection of the quadric $Q_k$ and some cubic hypersurface. Any line intersecting $S$ in at least 3 points is contained in $Q_k$, so the indeterminacy set is parametrized by the lines in $Q_k$, i.e. $\tilde{\Pi}_k \simeq \mathbb{P}^3$. The involution $\iota_k$ is biregular on the set $X \setminus \tilde{\Pi}_k$. Now, denote by $\epsilon_k: \widehat{X}_k \to X$ the blow-up of $X$ in $\tilde{\Pi}_k$. The exceptional divisor $E_k$ of this blow-up is isomorphic to the projectivization of the conormal bundle $\mathbb{P}\left(N^{*}_{X/\tilde{\Pi}_k}\right)$. Since $\Pi_k \simeq \mathbb{P}^3$ and is lagrangian, conormal is the same as tangent, and $$\mathbb{P}\left(N^{*}_{X/\tilde{\Pi}_k}\right) \simeq \big\{(x, H) \in \mathbb{P}^3 \times \left(\mathbb{P}^3\right)^{\vee}\big\}.$$
The map $\iota_k$ extends to a biregular involution of $\widehat{X}_k$ which maps $E_k$ to itself.

Define a birational automorphism $$f = \iota_2 \circ \iota_1: X \dashrightarrow X.$$
\begin{proposition} \label{prop: action}
    1) The action of $f^{*}$ on $\NS(X)$ in the basis $\{H_1, E, H_2\}$ is given by the matrix $$ M_1M_2 = \begin{pmatrix} 27 & 12 & 16 \\
-18 & -7 & -10 \\
-8 & -4 & -5 \end{pmatrix}.$$
2) $(f^n)^{*} = \left(f^*\right)^n$.
\end{proposition}
\begin{proof}
	The first part of the proposition is the equality $(\iota_2 \circ \iota_1)^{*} = \iota_1^{*} \circ \iota_2^{*}$, which holds because birational automorphisms of varieties with trivial canonical bundle contract no hypersurfaces. The second part is true by the same reason. \qed

\end{proof}

\smallskip

The eigenvalues of the operator $f^{*}: \operatorname{NS}(X) \to \operatorname{NS}(X)$ are $1, 7 - 4\sqrt{3}, 7 + 4\sqrt{3}$ with eigenvectors $$2H_1 - 7E + 2H_2,\, (-2 + \sqrt{3})H_1 + (1 - \sqrt{3})E + H_2,\, (-2 - \sqrt{3})H_1 + (1 + \sqrt{3})E + H_2$$ respectively. Since $7 + 4\sqrt{3} > 1$, the map $f$ is of infinite order.
\begin{proposition}
    The map $f$ is not induced by any birational automorphism of $S$. Moreover, it is not induced by any birational automorphism of a K3 surface $S'$ with $S^{[3]} \simeq (S')^{[3]}$.
\end{proposition}
\begin{proof}
    If $f$ is induced by an automorphism of $S'$, then $f^{*}(E') = E'$, where $2E'$ is an exceptional divisor of the morphism $(S')^{[3]} \to \operatorname{Sym}^3 S$. Then we get $E' = k(2H_1 - 7E + 2H_2)$ for some $k \in \mathbb{Q}$. However, this is not possible since $$-4 = (E')^2 = k^2 \cdot (2H_1 - 7E + 2H_2)^2 = -84k^2. \eqno \qed$$
\end{proof}
\begin{proposition} \label{prop: eff_inv}
No effective divisor on $X$ is invariant under $f$.
\end{proposition}
\begin{proof}
    The only invariant divisors on $X$ are multiples of $2H_1 - 7E + 2H_2$. Such a divisor is not effective since $$\iota_1^{*}\left(2H_1 - 7E + 2H_2\right) = -(2H_1 - 7E + 2H_2). \eqno \qed$$
\end{proof}

\section{The Mori cone} \label{mori}
In this section, we find the closure of the cone of effective curves $NE_1(X)$ (viewed as a part of the real N\'eron-Severi group by means of $q$) and the ample cone $\Amp(X)$ of $X$. Recall that by Kleiman's criterion, the divisor $D$ is ample if and only if it is strictly positive on the closure of $NE_1(X)$.

The main tool for the computation is the classification of extremal rays established in \cite{HT1} and \cite{HT2}.
\begin{lemma}
    $$\overline{\NE_1(X)} = \big\langle E, H_1 - \frac{3}{2}E, H_2 - \frac{3}{2}E, 2H_1 - H_2 - \frac{1}{2}E, 2H_2 - H_1 - \frac{1}{2}E \big\rangle_{\R_{\geq 0}}.$$ 
\end{lemma}
\begin{proof}
    Let $R = xH_1 + yH_2 + zE$ be a primitive element of $N_1(X)$.
    By \cite{HT1} and \cite{HT2} the ray generated by $R$ is an extremal ray of $\overline{\NE_1(X)}$ if and only if it satisfies one of the following properties
    \begin{itemize}
        \item $q(R, R) = -2, q(R, H_1) \geq 0$;

        \item $q(R, R) = -4, q(R, H_1) \geq 0$, and $x, y$ are both divisible by 2 but not by for 4;

        \item $q(R, R) = -4, q(R, H_1) \geq 0$, and $x, y$ are both divisible by 4;

        \item $q(R, R) = -12, q(R, H_1) \geq 0$, and $x, y$ are both divisible by 2;

        \item $q(R, R) = -36, q(R, H_1) \geq 0$, and $x, y$ are divisible by 4.
    \end{itemize}
    It can be easily checked that the classes given in the statement of this lemma satisfy one of the properties above.
    
    Now, assume that $R$ generates an extremal ray. If $q(R, H_1) = 0$, then $(x, y) = (4t, 3t)$ for some $t \in \Z$. Thus $(R, R) = -42t^2 - 4z^2$ which is possible only for $R = E$. Therefore, we may assume that $R$ intersects positively with both $H_1$ and $H_2$. Now, write
    \begin{multline*}\label{eq: NE}
        \big\langle E, H_1 - \frac{3}{2}E, H_2 - \frac{3}{2}E, 2H_1 - H_2 - \frac{1}{2}E, 2H_2 - H_1 - \frac{1}{2}E \big\rangle_{\R_{\geq 0}} = \\ = \big\{xH_1 + yH_2 + zE\,:\,x + 2y \geq 0, 2y + x \geq 0, 3x + 5y + 2z \geq 0, 3x + 5y + 2z \geq 0, \\ 3x + 3y + 2z \geq 0\big\}.
    \end{multline*}
    Let $R = xH_1 + yH_2 + zE$ be an extremal class.
    We need to check that numbers $x, y, z$ satisfy all five inequalities above. By symmetry, it is enough to prove inequalities $x + 2y \geq 0, 5x + 3y + 2z \geq 0$ and $3x + 3y + 2z \geq 0$. Observe that $\frac{3}{2}q(R, R) = t^2 - 7y^2 - 4z^2$, where $t = 3x + 4y > 0$.

    \begin{enumerate}  
        \item Suppose that $x + 2y < 0$. Thus, $t + 2y < 0$ and we get $$\frac{3}{2}q(R, R) = t^2 - 7y^2 - 4z^2 < -3y^2 - 4z^2.$$ Recall that $q(R, R) \in \{-2, -4, -12, -36\}$, so there are finitely many such $R$ and it can be shown by brute force that in fact there are no such $R$.

        \item Suppose that $3x + 5y + 2z < 0$. Thus, $t < -2z - y$ and we get $$\frac{3}{2}q(R, R) = t^2 - 7y^2 - 6z^2 < 6y^2 - 4yz - 2z^2 \leq -2y^2 - z^2.$$ By the same argument, it can be checked that there are no such $R$.

        \item Suppose that $3x + 3y + 2z < 0$. Thus, $t < 2z - y$ and we get $$\frac{3}{2}q(R, R) = t^2 - 7y^2 - 6z^2 < 6y^2 + 4yz - 2z^2 \leq -2y^2 - z^2.$$ This case is analogous to the previous one.  \qed

    \end{enumerate}

    \begin{remark} The classes generating the cone of curves are interpreted geometrically as follows.
\begin{itemize}
            \item It is well-known that the class of the ruling of the exceptional divisor of the Hilbert-Chow morphism is 
		    $\frac{1}{4}E$.

            \item According to \cite{HT2}, example 14, the classes $H_1 - \frac{3}{2}E, H_2 - \frac{3}{2}E$ are the classes of lines in projective subspaces $\tilde{\Pi}_1, \tilde{\Pi}_2$ defined in the section \ref{inv}.

            \item According to \cite{HT1}, example 4.11, the classes $2H_2 - H_1 - \frac{1}{2}E, 2H_1 - H_2 - \frac1{2}E$ are the classes of lines in projective subspaces $C_1^{[3]}, C_2^{[3]}$, where $C_1$ and $C_2$ are $(-2)$-curves having classes $2h_1 - h_2$ and $2h_2 - h_1$ respectively.

        \end{itemize}
    \end{remark}

	By duality, one obtains the following result.

    \begin{proposition} \label{prop: amp}
        The closure of the ample cone of $X$ is generated by the following classes
        \begin{multline*}
            \overline{\Amp(X)} = \big\langle 5H_1 - 2H_2, 5H_2 - 2H_1, 11H_1 - 3H_2 - 7E, 11H_2 - 3H_1 - 7E, 3H_1 + 3H_2 - 7E \big\rangle_{\mathbb{R}_{\geq 0}}.
        \end{multline*}
    \end{proposition}

	\begin{remark} From this description of the ample cone one can deduce that $X$ has only finitely many automorphisms.
			Indeed the projectivization of the cone of classes of positive square in the real N\'eron-Severi group can be seen as a model of the hyperbolic space $\H$. The projectivized ample cone of $X$ inside $\H$ is rational polyhedral of finite volume. According to \cite{AV-hyperb}, its quotient by $Aut(X)$ is one of finitely many pieces (of strictly positive volume) of the quotient orbifold of
			$\H$ by the monodromy group, hence $Aut(X)$ must be finite.
	\end{remark}

	This explains the necessity to work with a birational, rather than a biregular, automorphism, in spite of some technical difficulties. The following proposition shall be useful in dealing with them.

    \begin{proposition} \label{prop: inter}
        The intersection of $\tilde{\Pi}_1$ and $\tilde{\Pi}_2$ is empty.
    \end{proposition}
    \begin{proof}
        Let $C \subset S$ be the $(-2)$-curve in the class $2h_1 - h_2$. Then the image of $C$ under the embedding given by the class $h_1$ has degree $h_1 \cdot (2h_1 - h_2) = 4$. Assume that $C^{[3]} \cap \tilde{\Pi}_1 \neq \emptyset$. Then there is a trisecant line to a degree 4 rational curve $C \subset \mathbb{P}^4$. As a rational normal curve does not have a trisecant, it follows that $C$ must be contained in a hyperplane. Hence the class $h_1$ must be the sum of $2h_1 - h_2$ and an effective class. Since the class $h_2 - h_1$ is not effective, we conclude that the intersection $C^{[3]} \cap \tilde{\Pi}_1$ is empty. Now, observe that $\iota_1\left(C^{[3]}\right) = \tilde{\Pi}_2$ since $\iota_1^*(2H_1 - H_2 - \frac1{2}E) = H_2 - \frac{3}{2}E$. Since $\iota_1$ restricted to the complement of $\tilde{\Pi}_1$ is an automorphism, the intersection $\tilde{\Pi}_1 \cap \tilde{\Pi}_2$ is also empty. \qed  
    \end{proof}

    \begin{figure}[h]
\centering
\begin{tikzpicture}[main/.style = {fill, draw, circle, minimum width=2pt, inner sep=0pt}]
    \node[main] (x) at (0, 0) [label={-60: $E$}] {};
\node[main] (x1) at (4, 2) [label = {-60:$2H_1 - H_2 - \frac1{2}E$}] {};
\node[main] (x2) at (2, 4) [label = {60:$H_1 - \frac{3}{2}E$}] {};
\node[main] (y1) at (-4, 2) [label = {-120:$2H_2 - H_1 - \frac{1}{2}E$}] {};
\node[main] (y2) at (-2, 4) [label = {120: $H_2 - \frac{3}{2}E$}] {};

\draw[-] (0, 0) to (4, 2);
\draw[-] (4, 2) to (2, 4);
\draw[-] (2, 4) to (-2, 4);
\draw[-] (-2, 4) to (-4, 2);
\draw[-] (-4, 2) to (0, 0);

\draw (3, 0.6) node{$\left(5H_1 - 2H_2\right)^{\perp}$};
\draw (-2.9, 0.6) node{$\left(5H_2 - 2H_1\right)^{\perp}$};
\draw (0, 4.3) node{$\left(3H_1 + 3H_2 - 7E\right)^{\perp}$};
\draw (5, 3) node{$\left(11H_1 - 3H_2 - 7E\right)^{\perp}$};
\draw (-4.8, 3) node{$\left(11H_2 - 3H_1 - 7E\right)^{\perp}$};

\end{tikzpicture}
\caption{Transverse section of $\overline{NE_1(X)}$}
    \label{fig: loops}
\end{figure}
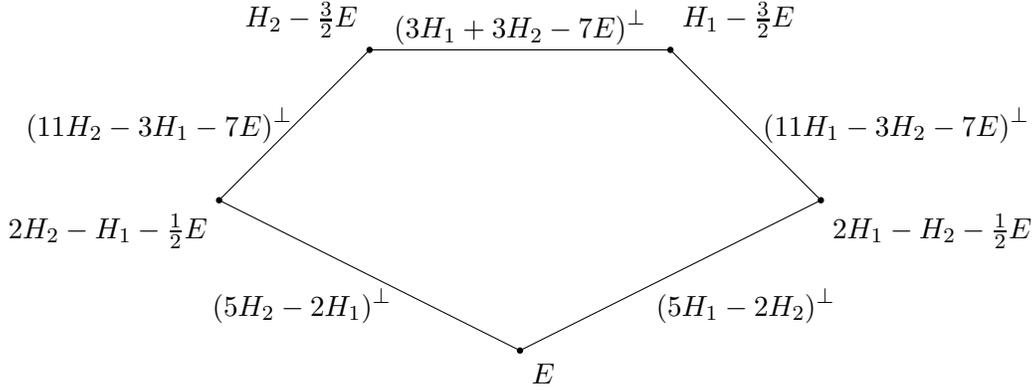
    
\end{proof}

\section{Four-dimensional case} \label{4-dim}

To prove the potential density, we shall investigate the set of $f$-iterates of a suitable $\mathbb{P}^3$ in $X$ and prove that it is Zariski-dense. We have already seen that it cannot be 5-dimensional (Proposition \ref{prop: eff_inv}).
In this section, we exclude the case of dimension four by proving the following proposition.

\begin{proposition} \label{prop: 4-dim}
	Let $X$ be an irreducible holomorphic symplectic sixfold\footnote{An irreducible holomorphic symplectic, or hyperk\"ahler, manifold is a compact K\"ahler manifold with $H^{2,0}$ generated by a single nowhere degenerate form $\sigma$. The punctual Hilbert scheme of a K3 surface is the best-known example.}.  
	If $Y \subset X$ is a subvariety with $\dim Y = 4$, then there is only a finite number of rationally connected lagrangian subvarieties $\Pi_i \subset X$ contained in $Y$.
\end{proposition}

Before proving the proposition, we recall some definitions.

\begin{definition}
    Let $M$ be a smooth compact complex manifold of dimension $n$. A Pfaff system $\mathcal{F}$ of codimension $r$ on $M$ is a non-trivial section $\omega \in \H^0(M, \Omega^r \otimes_{\mathcal{O}_M} L$, where $\Omega_M^r$ 
	denotes the sheaf of holomorphic $r$-forms 
	and $L$ is some holomorphic line bundle.

    An irreducible hypersurface $Z \subset Y$ is called $\omega$-invariant if $\omega|_{Z} = 0$.
\end{definition}

\begin{definition}
A first integral of the Pfaff system $\mathcal{F}$ defined by the form $\omega$ is a non-constant meromorphic map $f: M \to C$ for some curve projective smooth $C$ such that fibers of $f$ are $\omega$-invariant hypersurfaces.
\end{definition}

The following theorem, usually stated as the fact that a codimension-one holomorphic foliation with infinitely many closed leaves is a fibration, was proven in \cite{J} with some extra assumptions, satisfied for projective manifolds, and then in \cite{G}.

\begin{theorem}[\textbf{Jouanolou-Ghys}] Let $\mathcal{F}$ be a Pfaff system of codimension 1 on a connected compact complex manifold $M$, defined by the form $\omega \in \H^0(M, \Omega^1_M \otimes_{\mathcal{O}_M} L)$. If $\omega$ does not admit a first integral, then the number of $\omega$-invariant hypersurfaces is not greater than $$\dim \left[\H^0\left(M, \Omega^2_M \otimes_{\mathcal{O}_M} L \right)\big/\omega \wedge \H^0\left(M, \Omega^1_M\right)\right] + \rho(M) + 1.$$
\end{theorem}

For the case of an arbitrary codimension this theorem was generalized in \cite{CMS}.

\begin{theorem}
    Let $\mathcal{F}$ be a Pfaff system of codimension $r$ on a connected compact complex manifold $M$ defined by the form $\omega \in \H^0(M, \Omega^r_M \otimes_{\mathcal{O}_M} L)$. If $\omega$ does not admit a first integral, then the number of $\omega$-invariant hypersurfaces is not greater than $$\dim\left[H^0\left(M, \Omega^{r + 1}_M \otimes_{\mathcal{O}_M} L\right)\big/\omega\wedge H^0(M, \Omega^1_{\operatorname{cl}}) \right] + \dim H^1(M, \Omega^1_{\operatorname{cl}}) + r + 1,$$ where $\Omega^1_{\operatorname{cl}}$ denotes the sheaf of closed 1-form.
\end{theorem}

For our purpose, we use only the following corollary.

\begin{corollary}\label{cor: inv_hyp}
If the number of $\omega$-invariant hypersurfaces is infinite, then the Pfaff system defined by $\omega$ admits a first integral.
\end{corollary}

Now we are ready to prove Proposition \ref{prop: 4-dim}.

\medskip

\begin{proof}
    Resolving singularities if necessary we may assume that $Y$ is smooth and maps to $X$ birationally onto its image. Observe that the pull-back of the holomorphic symplectic form to $Y$ defines a Pfaff system of codimension 2. Since there are no non-trivial holomorphic forms on a rationally connected variety, the hypersurface $\Pi_i \subset Y$ is invariant for any $i$. By corollary \ref{cor: inv_hyp}, this Pfaff system admits first integral $f: Y \dashrightarrow C$.

    Blowing-up the variety $Y$ again if necessary, we get the following commutative diagram 
    \begin{center}
        \begin{tikzcd}
            \widehat{Y} \arrow[d, "\epsilon"] \arrow[rd, "\widehat{f}"] & \\
        Y \arrow[r, "f", dashed] & C
        \end{tikzcd}        
    \end{center}
    The morphism $\widehat{f}$ is flat since $C$ is a curve. By \cite{DF}, over the complex numbers the set of points $x \in X$ with rationally connected fiber $\widehat{f}^{-1}(x)$ is constructible. Since this set is infinite, it is Zariski-open. We get that there is an open subset $U \subset C$ with rationally connected fibers. Since there are no non-trivial holomorphic 2-forms on a rationally connected variety, the form determining the Pfaff system is a lift of some form on $C$. There are no such forms on $C$, because $\dim C = 1$. \qed
\end{proof}

\section{Potential density}

In this section, we finish the proof of our main result about potential density of rational points on $X$. Define by induction the  family of subvarieties $\{\Pi_k^{i}\} \, (i \in \{1, 2\}, k \in \mathbb{Z}_{\geq 0})$ $$\Pi^1_0 = \tilde{\Pi}_2, \, \Pi^2_n =\overline{\iota_1(\Pi^1_n \setminus \operatorname{Ind} \iota_1)}, \, \Pi^1_{n + 1} = \overline{\iota_2\left(\Pi^2_{n} \setminus \operatorname{Ind} \iota_2\right)}.$$  

Every embedding $\Pi^i_k \hookrightarrow X\,$ $(i \in \{1, 2\}, k \in \mathbb{Z}_{\geq 0})$ induce a map of Mori cones. Our goal is to describe the image of this map. For this purpose, it is convenient to work in bases $$L_i = H_{i + 1} - \frac{3}{2}E,\, K_i = 2H_{i + 1} - H_{i} - \frac{1}{2}E,\, E$$ 
(here, we set $i + 1 = 1$ for $i = 2$). The action of the morphism $(\iota_i)_*$ in these bases is as follows
\begin{gather*} (\iota_i)_* \left(L_i \right) = (\iota_i)_*\left(H_{i+1} - \frac{3}{2}E\right) = 2H_{i} - H_{i+1} - \frac{1}{2}E = K_{i + 1},\\
(\iota_i)_*\left(K_i\right) = (\iota_i)_*\left(2H_{i + 1} - H_{i} - \frac{1}{2}E\right) = 9H_{i} - 2H_{i+1} - \frac{15}{2}E = 5L_{i + 1} + 2K_{i + 1} + E, \\
(\iota_i)_*\left(E\right) = 4H_{i} - 5E = 4L_{i + 1} + E.
\end{gather*} 
Define the subcones 
$$J_1 = \big\langle L_1, K_1, E \big\rangle_{\mathbb{R}_{\geq 0}}, \, J_2 = \big\langle L_2, K_2, E \big\rangle_{\mathbb{R}_{\geq 0}} \subset NE_1(X).$$
It is clear that $(\iota_i)_*(J_i) \subset J_{i + 1}$.
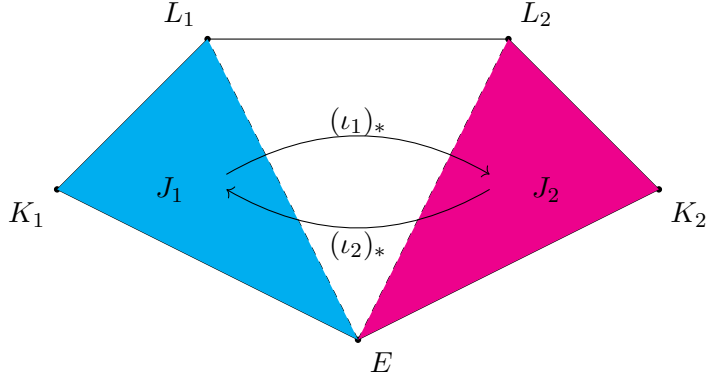
\begin{figure}[h]
\centering
\begin{tikzpicture}[main/.style = {fill, draw, circle, minimum width=2pt, inner sep=0pt}]
    \node[main] (x) at (0, 0) [label={-60: $E$}] {};
\node[main] (x1) at (4, 2) [label = {-60:$K_2$}] {};
\node[main] (x2) at (2, 4) [label = {60:$L_2$}] {};
\node[main] (y1) at (-4, 2) [label = {-120:$K_1$}] {};
\node[main] (y2) at (-2, 4) [label = {120: $L_1$}] {};

\draw[-] (0, 0) to (4, 2);
\draw[-] (4, 2) to (2, 4);
\draw[-] (2, 4) to (-2, 4);
\draw[-] (-2, 4) to (-4, 2);
\draw[-] (-4, 2) to (0, 0);
\draw[dashed] (0, 0) to (-2, 4);
\draw[dashed] (0, 0) to (2, 4);

\fill[cyan] (0, 0) -- (-2, 4) -- (-4, 2) -- cycle;
\fill[magenta] (0, 0) -- (2, 4) -- (4, 2) -- cycle;

\draw[->, bend left] (-1.75, 2.2) to (1.75, 2.2);
\draw[->, bend left] (1.75, 2) to (-1.75, 2);

\draw (0, 2.9) node{$(\iota_1)_*$};
\draw (0, 1.25) node{$(\iota_2)_*$};

\draw (-2.5, 2) node{$J_1$};
\draw (2.5, 2) node{$J_2$};

\end{tikzpicture}
\caption{Subcones $J_1$ and $J_2$}
    \label{fig: cone2}
\end{figure}

\begin{proposition} \label{prop: ind}
    1) For any integer $k \geq 1$, the indeterminacy set of the map $\iota_2: \Pi^2_{k - 1} \dashrightarrow \Pi^1_k$ is finite. For any integer $k \geq 0$ the image of the map $\NE_1(\Pi^1_k) \to \NE_1(X)$ is contained in the cone $J_1$.

    2) For any integer $k \geq 0$, the indeterminacy set of the map $\iota_2: \Pi^1_{k} \dashrightarrow \Pi^2_{k}$ is finite. For any integer $k \geq 0$ the image of the map $\NE_1(\Pi^2_k) \to \NE_1(X)$ is contained in the cone $J_2$.
\end{proposition}
\begin{proof}
	We prove the proposition by induction on $k$. We start with both parts of the proposition for $k = 0$ (this is the induction base for all statements but the very first, but the first statement for $k=1$ shall follow from part 2 for $k=0$). For the first part, it is enough to observe that $\NE_1(\Pi^1_0)$ is generated by $L_1$, and thus, is contained in $J_1$. For the second part, observe that $\iota_2$ is regular on $\Pi^1_0$ by proposition \ref{prop: inter}. Thus, $$\NE_1(\Pi^2_k) = (\iota_1)_*\left(\NE_1(\Pi^1_k)\right) \subset (\iota_1)_*(J_1) \subset J_2.$$

    Let us prove the induction step. Assume that the second part of the assertion is proved for $k$. 
	We show that the first part holds for $k + 1$. If the intersection $\Pi^2_k \bigcap \tilde{\Pi}_2$ is of dimension at least 1, then it contains an effective curve whose class is a positive multiple of $L_1$. However, $\NE_1(\Pi^2_k) \subset J_2$ which contradicts $L_1 \notin J_2$.

    It is clear that the set $\Pi^1_{k + 1} \setminus \iota_2(\Pi^2_{k} \setminus \operatorname{Ind}\iota_2)$ is contained in $\tilde{\Pi}_2$. Thus, any class of an effective curve in $\Pi^1_{k + 1}$ equals either a positive multiple of $L_1$ or $(\iota_2)_* C - \lambda L_1$, where $C$ is a class of some effective curve in $\Pi^2_k$ and $\lambda \geq 0$. We know that $C \in J_2$, so $(\iota_2)_* C 
\in J_1$. By proposition \ref{prop: amp}, the class of an effective curve has non-negative intersection with $5H_2 - 2H_1$. Suppose that $$(\iota_2)_* C - \lambda L_1 = \alpha L_1 + \beta K_1 + \gamma E.$$ It is evident that $\beta, \gamma \geq 0$. Using that $q(5H_2 - 2H_1, K_1) = q(5H_2 - 2H_1, E) = 0$ we get $(5H_2 - 2H_1, \alpha L_2) = 14 \alpha \geq 0$. Therefore, the class of any effective curve in $\Pi^1_k$ belongs to $J_1$.

The induction step from 
	1) for $k+1$ to 2) for $k+1$ is analogous. \qed
\end{proof}

\begin{proposition} \label{prop: periodic}
    The family of subvarieties $\{\Pi^1_k\}$ is not periodic. In other words, $$\Pi^1_n \neq \Pi^1_m$$ for $n \neq m$.
\end{proposition}
\begin{proof}
    It is enough to prove the proposition for $n = 0$. Assume the opposite, then $\Pi^1_m = \Pi^1_0$ for some positive $m$. Take a curve $C$ not passing through the indeterminacy set of the map $f^m$ (this set is finite by proposition \ref{prop: ind}). The class of the image curve is a positive multiple of $[C]$ (since both of them lie in $\Pi^1_0 \simeq \mathbb{P}^3$). Then $[C]$ is an eigenvector of the operator $(f^m)_*$ which contradicts the proposition \ref{prop: action}. \qed
\end{proof}

Finally, we are ready to prove the main theorem.

\begin{theorem}
    Rational points of $X$ are PD.
\end{theorem}
\begin{proof}
     By construction, there exists a finite field extension such that the rational points of all subvarieties $\Pi^1_k$ are dense in these subvarieties. So the Zariski-closure of rational points of $X$ over this extension contains the closure of the union $\bigcup \Pi^1_k$. Define $$Y = \overline{\bigcup_{n \in \mathbb{Z}_{\geq 0}} \Pi^1_n}.$$ By Proposition \ref{prop: periodic}, $\dim Y > 3$. By Proposition \ref{prop: 4-dim}, $\dim Y \neq 4$. If $\dim Y = 5$, then $Y$ corresponds to an effective $f$-invariant divisor which contradicts Proposition \ref{prop: eff_inv}. Therefore, $\dim Y = 6$, and $Y = X$. \qed
\end{proof}

\medskip

{\bf Acknowledgements:} The authors are grateful to Constantin Shramov and Jorge Vit\'orio Pereira for helpful discussions.
This study has been funded within the framework of the HSE University Basic Research Program.

\printbibliography

@article{A, 
title={On an automorphism of Hilb[2] of certain K3 surfaces}, 
volume={54}, 
number={1}, 
journal={Proceedings of the Edinburgh Mathematical Society}, 
author={E. Amerik}, 
year={2011}, 
pages={1–7}
}

@article{ABR,
author = {Amerik, E. and Bogomolov, F. and Rovinsky, M.},
year = {2011},
month = {11},
pages = {1819 - 1842},
title = {Remarks on endomorphisms and rational points},
number = {6},
volume = {147},
journal = {Compositio Mathematica},
}

@article{AC2008,
      title={Fibrations m{\'e}romorphes sur certaines vari{\'e}tes {\`a} fibr{\'e} canonique triviale}, 
      author={E. Amerik and F. Campana},
      year={2008},
      journal={Pure and Appl. Math. Quarterly, 4},
      number={2},
      pages={509-545},
       
}

@article{G,
    AUTHOR = {E. Ghys},
     TITLE = {A propos d'un th\'eor\`eme de J.-P. Jouanolou
              concernant les feuilles ferm\'ees des feuilletages
              holomorphes},
   JOURNAL = {Rend. Circ. Mat. Palermo (2)},
    VOLUME = {49},
      YEAR = {2000},
    NUMBER = {1},
     PAGES = {175--180},
      
}

@article{AV,
    author = {E. Amerik and C. Voisin},
    title = {{Potential density of rational points on the variety of lines of a cubic fourfold}},
    volume = {145},
    journal = {Duke Mathematical Journal},
    number = {2},
    publisher = {Duke University Press},
    pages = {379 -- 408},
    year = {2008},
}

@article {AV-hyperb,
    AUTHOR = {E. Amerik and M. Verbitsky},
     TITLE = {Hyperbolic geometry of the ample cone of a hyperk\"ahler
              manifold},
   JOURNAL = {Res. Math. Sci.},
  FJOURNAL = {Research in the Mathematical Sciences},
    VOLUME = {3},
      YEAR = {2016},
     PAGES = {Paper No. 7, 9},
      
}

@article{B1983,
  title={Vari{\'e}t{\'e}s K{\"a}hleriennes dont la premi{\`e}re classe de Chern est nulle},
  author={A. Beauville},
  journal={Journal of Differential Geometry},
  year={1983},
  volume={18},
  pages={755-782},
}

@article{PS1971,
  title={A Torelli theorem for algebraic surfaces of type K3},
  author={I. Pjateckii-Sapiro and I. Safarevic},
  journal={Mathematics of The USSR-izvestiya},
  year={1971},
  volume={5},
  pages={547-588},
}

@article{SD1974,
  title={Projective Models of K-3 Surfaces},
  author={B. Saint-Donat},
  journal={American Journal of Mathematics},
  year={1974},
  volume={96},
  pages={602},
}

@article{MP,
author = {Maulik, D. and Poonen, B.},
year = {2012},
month = {08},
pages = {},
title = {Néron–Severi groups under specialization},
volume = {161},
journal = {Duke Mathematical Journal},
}

@article{D1984,
        title={Un contre-exemple au théorème de Torelli pour les variétés symplectiques irréductibles },
        author={ O. Debarre },
        year={ 1984 },
        publisher={ Gauthier-Villars },
        volume={ 299 },
        pages={ 681-684 },
        number={ 14 },
}

@article{B1982,
    title={Some remarks on K{\"a}hler manifolds with $c_1 = 0$},
    author={A. Beauville},
    year = {1982},
    journal = {Classification of algebraic and analytic manifolds},
    volume = {39},
    pages = {1-26},
}

@article{HT-abelian,
  title={Abelian fibrations and rational points on symmetric products},
  author={B. Hassett and Y. Tschinkel},
  journal={International Journal of Mathematics},
  year={2000},
  volume={11},
  pages={1163-1176}
}

@article{CMS,
  title={Hypersurfaces Invariant by Pfaff Equations},
  author={M. Corr{\^e}a and L. G. Maza and M. G. Soares},
  journal={Communications in Contemporary Mathematics},
  year={2013},
  volume={17},
}

@article{C,
     author = {Campana, F.},
     title = {Orbifolds, special varieties and classification theory},
     journal = {Annales de l'Institut Fourier},
     pages = {499--630},
     publisher = {Association des Annales de l{\textquoteright}institut Fourier},
     volume = {54},
     number = {3},
     year = {2004},
}

@article{BT,
  title={Density of rational points on elliptic K3 surfaces},
  author={Bogomolov, FA and Tschinkel, Yu},
  journal={Asian Journal of Mathematics},
  volume={4},
  number={2},
  pages={351--368},
  year={2000},
  publisher={International Press of Boston}
}

@article{O,
author = {Oguiso, K.},
year = {2009},
pages = {101--111},
title = {A remark on dynamical degrees of automorphisms of hyperkähler manifolds},
volume = {130},
journal = {Manuscripta Mathematica},

}

@article{OG,
author = {O'Grady, K.},
year = {2005},
pages = {1223--1274},
title = {Involutions and linear systems on holomorphic symplectic manifolds},
volume = {15},
journal = {GAFA},

}

@article{HT1,
  title={Intersection Numbers of Extremal Rays on Holomorphic Symplectic Varieties},
  author={B. Hassett and Y. Tschinkel},
  journal={Asian Journal of Mathematics},
  volume={14},
  number={3},
  pages={303--322},
  year={2010},
  publisher={International Press of Boston}
}

@article{HT2,
  title={Extremal rays and automorphisms of holomorphic symplectic varieties},
  author={Hassett, B. and Tschinkel, Y.},
  journal={K3 surfaces and their moduli},
  pages={73--95},
  year={2016},
  publisher={Springer},
}

@article{J,
author = {Jouanolou, J.-P.},
year = {1978},
month = {10},
pages = {239-245},
title = {Hypersurfaces solutions d'une equation de Pfaff analytique},
volume = {232},
journal = {Mathematische Annalen},
}

@article{DF,
author = {de Fernex, T. and Fusi, D.},
year = {2012},
month = {01},
pages = {},
title = {Rationality in families of threefolds},
volume = {62},
journal = {Rendiconti del Circolo Matematico di Palermo},
}

\noindent {\sc Ekaterina Amerik \\
{\sc Laboratory of Algebraic Geometry,\\
National Research University HSE,\\
Department of Mathematics, 6 Usacheva Str. Moscow, Russia,}\\
also:\\
{\sc Universit\'e Paris-Saclay,\\
Laboratoire de Math\'ematiques d'Orsay,\\
Campus d'Orsay, B\^atiment 307, 91405 Orsay, France},\\
\tt  ekaterina.amerik@gmail.com},

\medskip

\noindent {\sc Mikhail Lozhkin \\
{\sc Laboratory of Algebraic Geometry,\\
National Research University HSE,\\
Department of Mathematics, 6 Usacheva Str. Moscow, Russia,}\\
\tt lozhkin.mixail@gmail.com}
\end{document}